\documentclass[reqno,11pt]{amsart}
\usepackage{amsmath,amssymb,latexsym,soul,cite,mathrsfs,accents}
\usepackage[dvipsnames]{xcolor}
\usepackage{color,enumitem,graphicx}
\usepackage[colorlinks=true,urlcolor=blue, citecolor=red,linkcolor=blue,
linktocpage,pdfpagelabels, bookmarksnumbered,bookmarksopen]{hyperref}
\usepackage[hyperpageref]{backref}
\usepackage[english]{babel}
\usepackage[left=2.6cm,right=2.6cm,top=2.9cm,bottom=2.9cm]{geometry}
\usepackage{tensor}
\newtheorem{theorem}{Theorem}
\newtheorem{lemma}[theorem]{Lemma}

\newtheorem{proposition}[theorem]{Proposition}
\newtheorem{remark}[theorem]{Remark}
\newtheorem{definition}[theorem]{Definition}

\newcommand{\innerthmname}{}

\theoremstyle{definition}

\makeatletter
\def\namedlabel#1#2{\begingroup
	#2%
	\def\@currentlabel{#2}%
	\phantomsection\label{#1}\endgroup
}

\def\XXint#1#2#3{{\setbox0=\hbox{$#1{#2#3}{\int}$ }
		\vcenter{\hbox{$#2#3$ }}\kern-.6\wd0}}

\newcommand*\owedge{\mathpalette\@owedge\relax}
\newcommand*\@owedge[1]{%
	\mathbin{%
		\ooalign{%
			$#1\m@th\bigcirc$\cr
			\hidewidth$#1\m@th\wedge$\hidewidth\cr
		}%
	}%
}
\makeatother

\newcommand{\ud}{\mathrm{d}}

\allowdisplaybreaks[3]

\title[Singular solutions of Hartree Equations]{Asymptotic behavior of solutions to a planar Hartree equation with isolated singularities}
\author[T. Feng]{Tao Feng}
\author[M. Yang]{Minbo Yang}
\author[X. Zhou]{Xianmei Zhou*}

\address[T. Feng]{School of Mathematical Sciences,
	Zhejiang Normal University
	\newline\indent
	321004, Jinhua-ZJ, People’s Republic of China}
\email{\href{mailto:fengtao@zjnu.edu.cn}{fengtao@zjnu.edu.cn}}

\address[M. Yang]{School of Mathematical Sciences,
	Zhejiang Normal University
	\newline\indent
	321004, Jinhua-ZJ, People’s Republic of China}
\email{\href{mailto:mbyang@zjnu.edu.cn}{mbyang@zjnu.edu.cn}}

\address[X. Zhou]{School of Mathematical Sciences,
	Zhejiang Normal University
	\newline\indent
	321004, Jinhua-ZJ, People’s Republic of China}
\email{\href{mailto:xmzhou@zjnu.edu.cn}{xmzhou@zjnu.edu.cn}}

\thanks{$^*$Minbo Yang was partially supported by the National Key Research and Development Program of China (No. 2022YFA1005700), National Natural Science Foundation of China (12471114) and Natural Science Foundation of Zhejiang Province (LZ22A010001).}
\thanks{*Corresponding author.}

\makeatletter
\@namedef{subjclassname@2020}{\textup{2020} Mathematics Subject Classification}
\makeatother

\begin{document}
	
	\begin{abstract}
		In this paper we investigate the isolated singularities of the Hartree type equation
		\begin{equation*}
			-\Delta u (x)= \left(\frac{1}{|x|^\alpha}*e^u\right)e^{u(x)}\quad \text{in }   B_{1}\setminus\{0\} ,
		\end{equation*}
		where $\alpha>0$, $\displaystyle \frac{1}{|x|^\alpha}*e^u\triangleq\int_{B_{1} \setminus \{0\}}\frac{e^u(y)}{|x-y|^\alpha}dy$, and the punctured ball $B_{1}\setminus\{0\}\subset \mathbb{R}^2$. Under the finite total curvature condition, by establishing a representation formula for singular solutions, we obtain the asymptotic behavior of the solutions near the origin. We also extend this asymptotic behavior results to the case with a general non-negative coefficient $K(x)$, and to the higher-order Hartree-type equations in any dimension $n \geq 3$.
	\end{abstract}
	
	\maketitle

	\numberwithin{equation}{section}
	\numberwithin{theorem}{section}
	
	\section{Introduction and main results}\label{section1}
	 	In the present paper, we are interested in the following planar Hartree type equation with an isolated singularity at the origin:
	 	\begin{equation}\label{second-order Hartree in d2}
	 	-\Delta u (x)= \left(\frac{1}{|x|^\alpha}*e^u\right)e^{u(x)}\quad \text{in }   B_{1}\setminus\{0\} \subset \mathbb{R}^{2},
	 \end{equation}
	 	where $\alpha>0$, $B_{1}\setminus\{0\} $ is the punctured ball in $\mathbb{R}^2$ and the convolution $\displaystyle \frac{1}{|x|^\alpha}*e^u\triangleq\int_{B_{1} \setminus \{0\}}\frac{e^u(y)}{|x-y|^\alpha}dy$.   We mainly consider the singular solutions $u \in C^{2}(B_{1} \setminus \left\{0\right\})$ for nonlocal equation \eqref{second-order Hartree in d2} under the finite total curvature condition
	 	\begin{equation}\label{second-order assumption}
	 		\int_{B_{1}\setminus\{0\}}\left(\frac{1}{|x|^\alpha}*e^u\right)e^{u(x)} \ud x <+\infty.
	 	\end{equation}

	  \medskip 
Elliptic equations involving the exponential nonlinearity and/or the isolated singularity at the origin have been widely studied  in recent decades,  since it plays an essential role in various geometric and physical problems.
	  Let $(M,g)$ be a complete Riemannian manifold and $K$ be a given function on $M$. The Gaussian curvature problem asks whether there exists a metric $g_{1}$, conformal to $g$ on $M$, whose Gaussian curvature is $\frac{K}{2}$? Writing $g_{1}=e^{u}g$ and specializing to the case $M=\mathbb{R}^{2}$ with $g$ the standard Euclidean metric and $K\equiv1$, the problem above is equivalent to solving the following classical elliptic equation (called the prescribed scalar curvature equation)
	  \begin{equation}\label{Q1}
	  	-\Delta u =e^{u} \quad \text{in }\mathbb{R}^{2}.
	  \end{equation}
	  Using the moving planes method, all the $C^{2}$ smooth solutions of \eqref{Q1} satisfying the finite total curvature condition $\int_{\mathbb{R}^{2}} e^{u(x)} \ud x <\infty$ were classified by Chen-Li \cite{Chen-Li1, Chen-Li2}. There are also some classification results for global smooth solutions of the higher-order Q-curvature equation
	  \begin{equation}\label{higher-order Q-curvature equation}
	  	(-\Delta)^{\frac{n}{2}} u=e^{u} \quad \text{in } \mathbb{R}^{n} \ ( n\geq 3).
	  \end{equation}
	  For even dimensions  $n$, under the assumption $\int_{\mathbb{R}^{n}} e^{u(x)} \ud x <\infty$ and the decay condition $u(x)=o(|x|^{2})$ as $|x| \to +\infty$, Lin \cite{Lin},  Wei and Xu \cite{Wei-Xu2} established a complete classification of all  solutions $u\in C^{n}(\mathbb{R}^{n})$ of \eqref{higher-order Q-curvature equation} for $n=4$ and $n=2m$  $(m\geq2)$, respectively. For the general integer $n$, Chang and Yang \cite{Chang-Yang}  classified the $C^{n}$ smooth solutions of equation \eqref{higher-order Q-curvature equation} with the decay condition $u(x)=\ln\frac{2}{1+|x|^{2}}+w\left(\zeta(x)\right) $ as $|x| \to +\infty$, where $w$ is a smooth function defined on the unit $n$-sphere $\mathbb{S}^{n}$. Moreover,  Lin \cite{Lin}, Xu \cite{xu2006}, Martinazzi \cite{Martinazzi} and Hyder \cite{Hyder} obtained the asymptotic behavior of solutions to \eqref{higher-order Q-curvature equation} at infinity. 
	  
	  \medskip
	 It is worth noting that by using the Kelvin transform, the asymptotic behavior of solutions to \eqref{Q1} at infinity directly shows the asymptotic behavior near the singularity $0$  for the corresponding singular problem
	 \begin{equation*}
	 	-\Delta u =|x|^{-4}e^{u} \quad \text{in }\mathbb{R}^{2}\setminus\left\lbrace 0 \right\rbrace .
	 \end{equation*}
	 Consequently, it is necessary for us to focus on the  asymptotic behavior of singular elliptic solutions with isolated singularities.
	 Over the past few decades, several significant results on the asymptotic behavior of singular solutions with isolated singularities have received significant attention in the literature.   
	 Note that if $\alpha=0$, the equation \eqref{second-order Hartree in d2} is reduced to the conformal Gaussian curvature equation
	 \begin{equation}\label{Q-curvature equation in d2}
	 	-\Delta u =e^{u} \quad \text{in }   B_{1} \setminus \left\{0\right\} \subset \mathbb{R}^{2}
	 \end{equation}
  with the finite total curvature condition
 \begin{equation}\label{Q-curvature assumption}
 	\int_{B_{1}\setminus\left\{0\right\}}e^{u(x)} \ud x <+\infty.
 \end{equation}
Equation \eqref{Q-curvature equation in d2} arises in the  prescribed Gaussian curvature problem on a Riemannian surface with conical singularities  \cite{Chen-Li3}. 
	 The earliest investigation for local singular solutions  $u \in C^{2}(B_{1} \setminus \left\{0\right\})$ of equation \eqref{Q-curvature equation in d2} was conducted by Chou and Wan in \cite{Chou-Wan}. They showed that there exists a constant $b>-2$ such that
	 \begin{equation}\label{Q-curvature asymptotic behavior}
	 	u(x)=b\ln|x|+\mathcal{O}(1)\quad \text{as }|x|\to 0
	 \end{equation}
	 by applying complex-analytic techniques. As mentioned in \cite{Chou-Wan}, the integrability condition \eqref{Q-curvature assumption} is necessary for $u$ being asymptotically radial.  Using the asymptotic estimate \eqref{Q-curvature asymptotic behavior}, Guo, Wan and Yang \cite{Guo-Wan-Yang} later derived an asymptotic expansion up to arbitrary orders of $u$ near the origin.
	 
	 \medskip
	 There are also some results for the higher-order conformal Q-curvature equation
	 \begin{equation}\label{Q-curvature equation in higher-order case}
	 	(-\Delta)^{\frac{n}{2}}u=e^{u} \quad \text{in }B_{1} \setminus \left\{0\right\} \subset\mathbb{R}^{n}
	 \end{equation}
	 with the important condition \eqref{Q-curvature assumption}, where  $n \geq 3$ is an integer. When $n=4$, under an additional decay condition at the origin
	  \begin{equation}\label{decay}
	 	u(x)=o(|x|^{-2}) \quad \text{as}\ |x|\rightarrow 0,
	 \end{equation} Guo and Liu \cite{Guo-Liu} established the asymptotic behavior of local singular solutions $u \in C^{4}(B_{1} \setminus \left\{0\right\})$  of \eqref{Q-curvature equation in higher-order case} near the origin via ODE methods, refined subsequently to a complete asymptotic expansion in \cite{Guo-Liu-Wan} by using spherical averaging and ODE analysis. For the general case $n \geq 3$, under a mild decay assumption near the origin weaker than \eqref{decay}, that is 	$$\int_{B_{r} \setminus \left\{0\right\}} |u(x)| \ud x=o\left(r^{n-2}\right)\quad \text{as}\ r \to 0,$$ Yang and Yang \cite{Yang-Yang} developed a PDE method to study the local singular solutions of \eqref{Q-curvature equation in higher-order case}. Specifically,  for odd integer $n \geq 3$, they consider the equation \eqref{Q-curvature equation in higher-order case} in the distributional sense, while for even dimensions $n$, the solution $u$ belongs to $ C^{n}(B_{1} \setminus \left\{0\right\})$.  Based on the characterization of isolated singularities for the Poisson equation,   the authors in \cite{Yang-Yang} derived  a representation formula for the local singular solutions of \eqref{Q-curvature equation in higher-order case}, and subsequently established its asymptotic behavior. They also extended the asymptotic behavior results for both \eqref{Q-curvature equation in d2} and \eqref{Q-curvature equation in higher-order case} to the general isolated singularity problem
	 \begin{equation}\label{kx case}
	 	(-\Delta)^{\frac{n}{2}}u=K(x)e^{u} \quad \text{in }B_{1} \setminus \left\{0\right\} \subset\mathbb{R}^{n},
	 \end{equation} 
  where  $K(x) \in L^{\infty}(B_{1})$ is non-negative. 
	 
	 \medskip
	It is well-known that for  $ n\geq 3$, the higher-dimensional analogue of equations \eqref{Q1} and \eqref{higher-order Q-curvature equation} is the following higher-order Yamabe equation
	\begin{equation}\label{higher-order Yamabe equation}
		\left(-\Delta\right)^{\frac{m}{2}}u=u^{\frac{n+m}{n-m}} \quad \text{in }B_{1} \setminus \left\{0\right\} \subset \mathbb{R}^{n},
	\end{equation}
	where  $m \in (0,n)$ is an integer.
	 This equation is related to the problem of finding a conformal Riemannian metric on a conformally flat $n$-dimensional manifold with a prescribed curvature.  
	 When $m=2$, Caffarelli, Gidas and Spruck \cite{Carffarelli-Gidas-Spruck} used a ``measure theoretic" variation of the Alexandrov reflection technique to obtain the asymptotic symmetry of local singular positive solutions $u \in C^{2}(B_{1} \setminus \left\{0\right\})$ of second order Yamabe equation. Later, Li \cite{LiC} established the same result by employing a simplified moving plane method. Subsequent works refined the asymptotics using spectral properties of the linearized Fowler operator, see \cite{Korevaar-Mazzeo-Pacard-Schoen}.  When $m \geq 2$ is an even integer,  Jin and Xiong \cite{Jin-Xiong} recently studied the higher-order Yamabe equation the asymptotic radial symmetry of singular positive solutions  of \eqref{higher-order Yamabe equation},
	 by establishing the representation formula of solutions of \eqref{higher-order Yamabe equation}. 
	 Moreover, for the radial symmetry and  classification of global singular positive solutions $u \in C^{m}\left(\mathbb{R}^{n} \setminus \left\{0\right\}\right)$ to equation \eqref{higher-order Yamabe equation} in $\mathbb{R}^{n} \setminus \left\{0\right\}$, we refer readers to the works of Gidas, Ni and Nirenberg \cite{Gidas-Ni-Nirenberg},  Carffarelli, Gidas and Spruck \cite{Carffarelli-Gidas-Spruck}, Schoen \cite{Schoen}, Chen and Li \cite{Chen-Li1},  Chen, Li and Ou \cite{Chen-Li-Ou2}, Lin \cite{Lin}, Wei and Xu \cite{Wei-Xu2},   Li \cite{LiY} as well as Du and Yang \cite{Du-Yang2}. 
	 Regarding results for the higher-order Yamabe equation with more general singular sets, see
	 \cite{Du-Yang2, Huang-Li-Zhou, Chen-Lin1, Chen-Lin2, Zhang}. Additionally, constructions of Delaunay-type solutions of higher-order Yamabe equation are provided in \cite{Frank-Konig,Guo-Huang-Wang-Wei,Andrade-Wei}. For the Yamabe equation with a general  prescribed curvature function $K(x)$, we refer to
	 \cite{Chen-Lin2, Chen-Lin3, Lin2, Wei-Xu1, Zhang, Du-Yang3}.
	 
	 \medskip
	 The analysis of isolated singularities in critical Hartree-type equations is driven by motivations from both physical models and fundamental analytic inequalities, see \cite{Lieb2,Moroz-Penrose-Tod,Nazin}. A prototypical example is the following critical Hartree equation
	\begin{equation}\label{critical hartree equation}
		-\Delta u=(\mathcal{R}_\alpha\ast u^{\frac{2n-\alpha}{n-2}})u^{\frac{n+2-\alpha}{n-2}} \quad \text{in }\Omega\subset \mathbb{R}^{n},
	\end{equation}
	where $\mathcal{R}_\alpha(x)\ast u^{\frac{2n-\alpha}{n-2}}:=\int_{\Omega}\frac{u^{\frac{2n-\alpha}{n-2}}(y)}{|x-y|^\alpha}dy$ with $n\geqslant 3$ and $\alpha\in(0,n)$. If $\Omega=\mathbb{R}^{n}$, a complete classification of the global classical positive solutions to equation \eqref{critical hartree equation} has been established in the works of Miao et al. \cite{Miao-Wu-Xu}, Du and Yang \cite{Du-Yang1}, Gao et al. \cite{Gao-Yang}, and Le \cite{Le}, with related results also in \cite{Guo-Hu-Peng-Shuai}. 
	Recently, Andrade, Feng, Piccione, and Yang \cite{Andrade-Feng-Piccione-Yang} studied the singular solutions to \eqref{critical hartree equation},
	by establishing the integral representation of  positive solutions of \eqref{critical hartree equation}. On the one hand, when $\Omega= \mathbb{R}^{n} \setminus \left\{0\right\}$, the authors in \cite{Andrade-Feng-Piccione-Yang}  obtained the  radial symmetry of global positive solutions $u \in C^{2}\left(\mathbb{R}^{n} \setminus \left\{0\right\}\right) \cap L^{\frac{2n-\alpha}{n+2}}(\mathbb{R}^{n})$ of \eqref{critical hartree equation}. On the other hand,  when $\Omega= B_{r} \setminus \left\{0\right\}$ with $0<r<\infty$, they derive the asymptotic behavior of local positive solutions $u \in C^{2}\left(B_{r} \setminus \left\{0\right\}\right)\cap L^{\frac{2n-\alpha}{n+2}}(B_{r})$ of  \eqref{critical hartree equation}, that is,
	\begin{equation*}
		u(x)=\mathcal{O}(|x|^{\frac{2-n}{2}}) \quad \text{as } |x| \to 0.
	\end{equation*}
	and further obtained the asymptotic radial symmetry for the singluar solutions of \eqref{critical hartree equation} near the origin.
	
		All the results for \eqref{critical hartree equation} mentioned above concern the case $n \geq 3$. For the special case $n=2$, under some integrability conditions, there are some classification and asymptotic behavior results of global solutions $u \in C^{2}(\mathbb{R}^{n})$ of related Hartree-type equations with exponential nonlinearity, see \cite{Guo-Peng,Niu,Yang-Yu,Gluck}. However, to our knowledge, there are few results concerning the Hartree-type equation \eqref{second-order Hartree in d2} with an isolated singularity in dimension two. 

	\medskip
	It should be noted that, unlike the prescribed scalar curvature equation \eqref{Q1}, the right-hand side of equation \eqref{second-order Hartree in d2} is a nonlocal term. Consequently, ODE-based techniques in \cite{Guo-Liu} is extremely difficult to apply when analyzing the asymptotic behavior of singular solutions to \eqref{second-order Hartree in d2}. A natural question then arises: can the PDE-based approach be adapted to the study of \eqref{second-order Hartree in d2}? Moreover, does the asymptotic behavior \eqref{Q-curvature asymptotic behavior} remain valid for \eqref{second-order Hartree in d2}? Recalling that, in this paper, we require $\alpha>0$, and the finite total curvature condition \eqref{second-order assumption} holds.  Inspired by \cite{Yang-Yang}, we establish a representation formula for local singular solutions of \eqref{second-order Hartree in d2} and subsequently derive their asymptotic behavior at the origin. Our first main result is stated below.
	\begin{theorem}\label{secondorderresult}
		Suppose that $u\in C^{2}\left(B_{1}\setminus\{0\}\right)$ is a solution of \eqref{second-order Hartree in d2} and $\alpha>0$. If $u$ satisfies \eqref{second-order assumption}, then there exists a constant $b>-2$ such that
		\begin{equation*}
			u(x)=v(x)+h(x)+b\ln|x| \quad\text{for }x\in B_{1} \setminus\{0\},
		\end{equation*}
		where $h \in C^{\infty}(B_{1})$ is a solution of $-\Delta h =0$ in $B_{1}$ and $v$ is defined by 
		\begin{equation*}
			v(x)=\frac{1}{2\pi}\int_{B_{1}\setminus\left\{0\right\}}\ln\frac{5}{|x-y|}\left(\frac{1}{|y|^\alpha}*e^u\right)e^{u(y)}\ud y.
		\end{equation*}
		Moreover, if $\alpha \in \left(0,2\right)$ and $b > \frac{\alpha-4}{2}$, then there exists a H$\ddot{o}$lder continuous function $\varphi \in C_{loc}^{\gamma}(B_{1}) $  with $\gamma\in(0,1)$  such that
		\begin{equation}\label{u(x)=varphi+blnx}
			u(x)=\varphi + b\ln|x| \quad \text{near the origin}.
		\end{equation}
	\end{theorem}

	More generally, we will consider the following Hartree type equations in higher-order case
	\begin{equation}\label{higher-order Hartree}
		(-\Delta)^{\frac{n}{2}} u (x)= \left(\frac{1}{|x|^\alpha}*e^u\right)e^{u(x)}\quad \text{in }   B_{1}\setminus\{0\} \subset \mathbb{R}^{n},
	\end{equation}
	where $n\geq3$ is an integer. Obviously, equation \eqref{higher-order Hartree} is a polyharmonic equation in even dimensions and a higher-order fractional equation in odd dimensions. 
	
	\medskip
	Based on the analysis in the proof of Theorem \ref{secondorderresult}, we can again obtain a representation formula for the  singular solutions of \eqref{higher-order Hartree}. However, compared with the two dimensional case ($n=2$), for the higher-order Hartree-type equations in  dimensions $n\geq3$, the following additional decay conditions on the solutions of \eqref{higher-order Hartree} are required to eliminate the higher-order derivative terms:
	\begin{equation}\label{o(r^{n-2})}
		\int_{B_{r} \setminus \left\{0\right\}} |u(x)| \ud x=o\left(r^{n-2}\right) \quad \text{as }r \to 0.
	\end{equation} 

	We first consider the even-dimensional case.  In this case, a fundamental solution of $(-\Delta)^{\frac{n}{2}}$ is given by $\phi(x) = c_{n}\ln\frac{5}{|x|}$, where $c_{n}>0$ is a constant.  The second result of this paper is 
	\begin{theorem}\label{even result case}
		Suppose that $n \geq 4$ is an even integer and $\alpha>0$. Let $u \in C^{n}\left(B_{1}\setminus\{0\}\right)$ be a solution of \eqref{higher-order Hartree}. If $u$ satisfies \eqref{second-order assumption} and the decay condition \eqref{o(r^{n-2})}, then there exists a constant $b>-n$ such that
		\begin{equation*}
			u(x)=v(x)+h(x)+b\ln|x| \quad\text{for }x\in B_{1} \setminus\{0\},
		\end{equation*}
		where $h \in C^{\infty}(B_{1})$ is a solution of $(-\Delta)^{\frac{n}{2}} h =0$ in $B_{1}$ and $v$ is defined by 
		\begin{equation}\label{v higherorder}
			v(x)=c_{n} \int_{B_{1}\setminus\left\{0\right\}}\ln\frac{5}{|x-y|}\left(\frac{1}{|y|^\alpha}*e^u\right)e^{u(y)}\ud y, 
		\end{equation}
		Moreover, if $\alpha \in \left(0,n\right)$ and $b > \frac{\alpha-2n}{2}$, then there exists a H$\ddot{o}$lder continuous function $\varphi \in C_{loc}^{\gamma}(B_{1}) $  with $\gamma\in(0,1)$ such that
		\begin{equation}\label{asym}
			u(x)=\varphi(x) + b\ln|x| \quad \text{near the origin}.
		\end{equation}
	\end{theorem}
	Next, we consider the odd-dimensional case. Notice that when $n\geq3$ is an odd integer, $(-\Delta)^{\frac{n}{2}}$ is a nonlocal integral operator. Therefore, in this dimension, we study equation \eqref{higher-order Hartree} in the distributional sense. 
	\begin{definition}
		Let $n\geq3$ be an odd integer and $\alpha>0$. We say that $u \in C^{n}\left(B_{1} \setminus \left\{0\right\}\right)$ is a solution of \eqref{higher-order Hartree}, if the following conditions hold:\\
		(i) The function $u$ is an element of $\mathcal{L}_{\frac{n}{2}}(\mathbb{R}^{n}):=\left\{u\in L^{1}_{loc}(\mathbb{R}^{n}): \int_{\mathbb{R}^{n}}\frac{|u(x)|}{1+|x|^{2n}} \ud x <\infty\right\}$;\\
		(ii) $\left(\frac{1}{|x|^\alpha}*e^u\right)e^{u(x)} \in L^{1}_{loc}(B_{1}\setminus\left\{0\right\})$;\\
		(iii) For every test function $\varphi \in C^{\infty}_{c}(B_{1}\setminus\left\{0\right\})$, the following integral equality holds:
		\begin{equation*}
			\int_{\mathbb{R}^{n}} u(x)\left(-\Delta\right)^{\frac{n}{2}} \varphi(x) \ud x = \int_{B_{1}\setminus\left\{0\right\}} \left(\frac{1}{|x|^\alpha}*e^u\right)e^{u(x)} \varphi(x) \ud x.
		\end{equation*}
	\end{definition}
	Then we have the following result.
	\begin{theorem}\label{odd case result}
		Suppose that $n \geq3$ is an odd integer and $\alpha>0$. Let $u \in \mathcal{L}_{\frac{n}{2}}(\mathbb{R}^{n})\cap C^{n} (B_{1} \setminus \left\{0\right\})$ be a solution of \eqref{higher-order Hartree}. If $u$ satisfies \eqref{second-order assumption} and \eqref{o(r^{n-2})}, then there exists a constant $b>-n$ such that
		\begin{equation*}
			u(x)=v(x)+h(x)+b\ln|x| \quad\text{for }x\in B_{1} \setminus\{0\},
		\end{equation*}
		where $v$ is defined by \eqref{v higherorder} and $h \in \mathcal{L}_{\frac{n}{2}} (\mathbb{R}^{n})\cap C^{\infty}(B_{1})$ is a solution of $(-\Delta)^{\frac{n}{2}} h =0$ in $B_{1}$. Moreover, if $\alpha \in \left(0,n\right)$ and $b > \frac{\alpha-2n}{2}$, then there exists a H$\ddot{o}$lder continuous function $\varphi \in C_{loc}^{\gamma}(B_{1})$ with $\gamma\in(0,1)$ such that
		\begin{equation}\label{varphi}
			u(x)=\varphi(x) + b\ln|x| \quad \text{near the origin}.
		\end{equation}	
	\end{theorem}
	\begin{remark}[Asymptotic behavior]\label{remark asymptotic behavior}
		 For $\alpha \in \left(0,n\right)$ and $b > \frac{\alpha-2n}{2}$, we can describe the asymptotic behavior of the solutions to \eqref{second-order Hartree in d2} and \eqref{higher-order Hartree} at the origin. In fact, 
		by \eqref{varphi} and $\varphi \in C^{\gamma}_{loc} (B_{1})$ for some $\gamma\in(0,1)$, there exist constants $C_{1}$ and $C_{2}$ such that
		\begin{equation}\label{asymptotic}
			C_{1}+b\ln|x|\leq u(x) \leq C_{2} +b\ln|x|  \quad \text{near the origin}.
		\end{equation}
	\end{remark}
	\begin{remark}[Valid range of $b$]\label{Remark b}
		In Theorem~\ref{secondorderresult}, \ref{even result case} and \ref{odd case result}, for   $\alpha \in (0,n)$ and $n\geq2$, the solution exhibits the asymptotic behavior given in \eqref{u(x)=varphi+blnx}, \eqref{asym} and \eqref{varphi} near the origin only when
		\begin{equation*}
			b>\frac{\alpha-2n}{2}.
		\end{equation*} 
	 To be more precise, assume that there holds $u(x) =\varphi(x) +b\ln|x|$ near the origin with $\varphi \in C^{\gamma}_{loc} (B_{1})$ for some $0<\gamma<1$, then \eqref{asymptotic} also holds. Hence, for sufficiently small $r>0$, we have
		\begin{align*}
			\int_{B_{1}\setminus\{0\}}\left(\frac{1}{|x|^\alpha}*e^u\right)e^{u(x)} \ud x \geq C\int_{B_{r}}\left(\int_{B_{r}}\frac{|y|^{b}}{r^{\alpha}} \ud y\right) |x|^{b} \ud x,
		\end{align*}
		which implies the finite total curvature condition \eqref{second-order assumption} fails whenever $b \leq \frac{\alpha-2n}{2}$. Therefore, it must have $b >\frac{\alpha-2n}{2}$.
		 
	\end{remark}
	\begin{remark}[General coefficient $K(x)$]
		Replace equations \eqref{second-order Hartree in d2} and \eqref{higher-order Hartree} by equations
		\begin{equation}\label{K(x) second-order Hartree}
			-\Delta u (x)=K(x) \left(\frac{1}{|x|^\alpha}*e^u\right)e^{u(x)}\quad \text{in }   B_{1}\setminus\{0\} \subset \mathbb{R}^{2}
		\end{equation}
		and
		\begin{equation}\label{k(x)equation}
			(-\Delta)^{\frac{n}{2}} u (x)=K(x) \left(\frac{1}{|x|^\alpha}*e^u\right)e^{u(x)}\quad \text{in }   B_{1}\setminus\{0\} \subset \mathbb{R}^{n}, \ n\geq3
		\end{equation}
		respectively, where the only difference is the presence of an additional non-negative coefficient $K(x) \in L^{\infty}(B_{1})$. Then, Theorem ~\ref{secondorderresult}, \ref{even result case},~\ref{odd case result} and Remark ~\ref{Remark b} remain valid in this setting. In these cases, the expression of $v$ becomes
		\begin{equation}\label{v(x)}
			v(x)=c_{n} \int_{B_{1}\setminus\left\{0\right\}}\ln\frac{5}{|x-y|}K(y)\left(\frac{1}{|y|^\alpha}*e^u\right)e^{u(y)}\ud y,
		\end{equation}
		where $c_{n}>0$ is a constant and $c_{2}=\frac{1}{2\pi}$. It is worth noting that when $K(x) \equiv1$, equations \eqref{K(x) second-order Hartree} and \eqref{k(x)equation} reduce to the original equations \eqref{second-order Hartree in d2} and \eqref{higher-order Hartree} respectively. Therefore, in the proofs below, we shall work directly with equations \eqref{K(x) second-order Hartree} and \eqref{k(x)equation}.
	\end{remark}
	 Inspired by \cite{Yang-Yang},  a key step in our approach is to obtain a representation formula for the solutions of differential equations \eqref{K(x) second-order Hartree} and \eqref{k(x)equation}. This method not only avoids the failure of ODE-based techniques caused by the convolution term, but also bypasses the breakdown of the maximum principle induced by the nonlocality of the equation.  This idea plays a significant role in analyzing the properties of solutions to  higher-order equations with isolated  singularities, as showed in \cite{Du-Yang2,Jin-Xiong}.

\medskip
 This paper is organized as follows. In Section 2, we firstly deduce a representation formula for singular solutions of equations \eqref{K(x) second-order Hartree} and \eqref{k(x)equation} in even dimensions $n\geq3$, and then prove Theorem \ref{secondorderresult} and Theorem \ref{even result case}, by using the elliptic regularity theory.   In Section 3, we generalize the asymptotic behavior at the origin for the solutions to \eqref{k(x)equation} in odd dimensions  $n\geq3$, and prove Theorem \ref{odd case result}. In this paper, $c$, $C$ will be used to denote different constants.

	\section{Asymptotic behavior at isolated singularities in even dimension case }
	In this section,   we focus on the singular solutions of equation \eqref{k(x)equation} with even dimensions $n\geq2$ for simplicity. First, under the finite total curvature assumption \eqref{second-order assumption}, we derive a crucial representation formula for solutions of \eqref{k(x)equation} expressed in term of  $v$. Subsequently, for $n=2$ and for even $n\geq4$,  we verify the local H\"older continuity of $v$ under necessary conditions, which allows us to obtain the asymptotic behavior of the solution at the origin.
	\subsection{Representation formula}
	Motivated by \cite{Yang-Yang} and based on the characterization of isolated singularities for the Poisson equation, we now present the representation formula for solutions of \eqref{k(x)equation} in the case of even dimension $n\geq2$. 
	\begin{proposition}[Representation Formula]\label{representation formula}
		Let $n\geq2$ be a positive even integer and $\alpha>0$.
		Suppose that $u \in C^{n}\left(B_{1}\setminus\{0\}\right)$ is a solution of \eqref{k(x)equation}, where $K \in L^{\infty}(B_{1})$ is non-negative and $u$ satisfies \eqref{second-order assumption}, then there exists constants $a_{\beta} \in \mathbb{R}$ with $|\beta|\leq n-1$, such that
		\begin{equation}
			u(x)=v(x)+h(x)+\sum_{|\beta|\leq n-1} a_{\beta}D^{\beta}\phi(x)  \quad\text{for }x\in B_{1} \setminus\{0\},
		\end{equation}
		where $\phi(x) = c_{n} \ln\frac{5}{|x|}$ is a fundamental solution of $(-\Delta)^\frac{n}{2}$ in $\mathbb{R}^{n}$, $h\in C^{\infty}(B_{1})$ is a solution of $(-\Delta)^{\frac{n}{2}}h=0$ in $B_{1}$ and $v$ is given by \eqref{v(x)}.
	\end{proposition}
	
	\begin{proof}
		Define $u^{+}(x) :=\max\left\{u(x), 0\right\}$. It is easy to verify that $u^{+}(x) \leq e^{{u}(x)}$. Thus, we have
		\begin{align*}
			\int_{B_{1}\setminus\{0\}}K(x)\left(\frac{1}{|x|^\alpha}*e^u\right)e^{u(x)} \ud x& \geq C \int_{B_{1}\setminus \{0\}}\left(\int_{B_{1}\setminus \{0\}}\frac{u^{+}(y)}{|x-y|^{\alpha}} \ud y\right) u^{+}(x) \ud x \\
			& \geq \frac{C}{2^{\alpha}}\left(\int_{B_{1} \setminus \{0\}}u^{+}(x) \ud x\right)^{2}.
		\end{align*}
		This, together with \eqref{second-order assumption}, implies that
		\begin{equation*}
			\int_{B_{1} \setminus \{0\}}u^{+}(x) \ud x <+\infty.
		\end{equation*}
		Applying the results in \cite[Theorem 1.4]{Yang-Yang} with $s=0$, we conclude the proof of this proposition.
	\end{proof}
	\subsection{Two-dimensional case}	
	For the case $n=2$, a fundamental solution of $-\Delta$ in $\mathbb{R}^{2}$ is given by $\phi(x)=\frac{1}{2\pi}\ln\frac{5}{|x|}$. We will give the following specific representation for solutions of equation \eqref{K(x) second-order Hartree}. 
	\begin{lemma}\label{lemma2.2}
		Let $\alpha>0$. Suppose that $u \in C^{2}\left(B_{1}\setminus\{0\}\right)$ is a solution of \eqref{K(x) second-order Hartree} satisfying condition \eqref{second-order assumption}, and $K \in L^{\infty}(B_{1})$ is non-negative. Then u decomposes as
		\begin{equation*}
			u(x)=v(x)+h(x)+ a_{0}\phi(x)  \quad\text{for }x\in B_{1} \setminus\{0\},
		\end{equation*}
		where  $v$ is given by \eqref{v(x)},
	 $h \in C^{\infty}(B_{1})$ is a harmonic function in $B_{1}$ satisfying $-\Delta h =0$, and the constant $a_{0}$ satisfies the constraint  $a_{0}<4\pi$.
	\end{lemma}
	\begin{proof}
		Using proposition~\ref{representation formula} with $n=2$, we obtain that
		\begin{equation}\label{secondorderres1}
			u(x)=v(x)+h(x)+\sum_{|\beta|\leq 1} a_{\beta}D^{\beta}\phi(x)  \quad\text{for }x\in B_{1} \setminus\{0\},
		\end{equation}
		where $v$ is given by \eqref{v(x)}. We will prove that $a_{\beta}=0$ for $|\beta|=1$. By \eqref{second-order assumption} and \eqref{secondorderres1}, we obtain that
		\begin{equation}\label{second-order assumption1}
			\int_{B_{\frac{1}{2}} \setminus \{0\}}\left(\int_{B_{\frac{1}{2}} \setminus \{0\}}\frac{e^{v(y)}e^{h(y)}e^{\sum_{|\beta|\leq1}a_{\beta}D^{\beta}\phi(y)}}{|x-y|^{\alpha}}\ud y\right)e^{v(x)}e^{h(x)}e^{\sum_{|\beta|\leq1}a_{\beta}D^{\beta}\phi(x)} \ud x<+\infty.
		\end{equation}
		Since $h \in C^{\infty}(B_{1}(0))$, we get that $C_{1}< e^h <C_{2}$ in $B_{1/2}$ for two constants $C_{1}, C_{2}>0$. By the nonnegativity of $K$, then $e^{v} \geq 1 $ in $B_{1/2} \setminus \{0\}$. For $x,y \in B_{1/2}$, we have $\frac{1}{|x-y|^\alpha} \geq 1$. Hence, from \eqref{second-order assumption1}, we yields
		\begin{equation*}
			\int_{B_{\frac{1}{2}} \setminus \{0\}} e^{\sum_{|\beta|\leq 1} a_{\beta}D^{\beta}\phi(x)} \ud x<+\infty.
		\end{equation*}
		Note that $D^{(i,j)} \phi$ represents the derivative of $\phi$, taken $i$ times with respect to $x_{1}$ and $j$ times with respect to $x_{2}$. Thus, based on the expression of the fundamental solution of $-\Delta$ in $\mathbb{R}^{2}$, we obtain that
		\begin{equation}\label{second-order assumption2}
			\int_{B_{\frac{1}{2}} \setminus \{0\}}e^{a_{0}\frac{1}{2\pi}\ln \frac{5}{|x|}+a_{(1,0)}\left(-\frac{x_{1}}{2\pi|x|^{2}}\right)+a_{(0,1)}\left(-\frac{x_{2}}{2\pi|x|^{2}}\right)} \ud x <+\infty.
		\end{equation}
		
Consider the polar coordinates
		\[
		\begin{cases}
			x_1 = r \cos \theta, \\
			x_2 = r \sin \theta,
		\end{cases}
		\quad 0 \le r < 1, \ -\pi \le \theta \le \pi.
		\]
		Now, we show that if $a_{(0,1)}\leq0$, then $a_{(1,0)}=0$. Let
		\begin{equation*}
			D_{1}=\left\{(x_{1},x_{2})\in B_{1/2}\setminus\{0\}\bigg| 0\leq r <\frac{1}{2}, 0 \leq \theta \leq \frac{\pi}{4} \right\}. 
		\end{equation*}
		By \eqref{second-order assumption2}, we get that
		\begin{equation*}
			\int_{0}^{\frac{1}{2}} \int_{0}^{\frac{\pi}{4}}\left(\frac{5}{r}\right)^{\frac{1}{2\pi}a_{0}}e^{\frac{1}{2\pi}\left[a_{(1,0)}\left(-\frac{\cos\theta}{r}\right)+a_{(0,1)}\left(-\frac{\sin\theta}{r}\right)\right]}r \ud \theta \ud r < \infty.
		\end{equation*}
		Since $a_{(0,1)}\leq 0 $ and $0<\theta<\frac{\pi}{4}$, we obtain that
		\begin{equation*}
			\int_{0}^{\frac{1}{2}} \int_{0}^{\frac{\pi}{4}}\left(\frac{5}{r}\right)^{\frac{1}{2\pi}a_{0}}e^{\frac{1}{2\pi}\left[a_{(1,0)}\left(-\frac{\cos\theta}{r}\right)\right]}r \ud \theta \ud r < \infty.
		\end{equation*}
		Using $\cos \theta \geq \frac{\sqrt{2}}{2}$ in $D_{1}$ and the following equality
		\begin{equation*}
			\int_{0}^{\frac{1}{2}} r^{a}e^{\frac{b}{r}} \ud r =+\infty \quad \forall a \in \mathbb{R}, b>0,
		\end{equation*}
		we have $a_{(1,0)}\geq 0$. On the other hand, we set
		\begin{equation*}
			D_{2}=\left\{(x_{1},x_{2})\in B_{1/2}\setminus\{0\}\bigg| 0\leq r <\frac{1}{2}, \frac{3\pi}{4} \leq \theta \leq \pi \right\}. 
		\end{equation*}
		Then we can also get $a_{(1,0)}\leq 0$. Hence, if $a_{(0,1)}\leq0$ holds true, then $a_{(1,0)}=0$ holds as well.
		
	 Similarly, if $a_{(0,1)}\geq0$, we can also deduce that $a_{(1,0)}=0$, by analyzing the integral in 
		\begin{equation*}
			D_{3}=\left\{(x_{1},x_{2})\in B_{1/2}\setminus\{0\}\bigg| 0\leq r <\frac{1}{2}, -\frac{\pi}{4} \leq \theta \leq 0 \right\}
		\end{equation*}
		and
		\begin{equation*}
			D_{4}=\left\{(x_{1},x_{2})\in B_{1/2}\setminus\{0\}\bigg| 0\leq r <\frac{1}{2}, -\pi \leq \theta \leq -\frac{3\pi}{4} \right\}.
		\end{equation*}
		Thus, we obtain $a_{(1,0)}=0$. Moreover, we can also get $a_{(0,1)}=0$. 
		
		Combined with \eqref{second-order assumption2}, we have
		\begin{equation*}
			\int_{B_{\frac{1}{2}} \setminus \{0\}} \left(\frac{5}{\sqrt{x_{1}^{2}+x_{2}^2}}\right)^{\frac{1}{2\pi}a_{0}} \ud x_{1} \ud x_{2} <\infty,
		\end{equation*}
		which yields that $a_{0}<4\pi$.
	\end{proof}
	Next, we proceed to prove Theorem~\ref{secondorderresult}. We establish the local H\"older continuity of $v$, from which the asymptotic behavior of $u$ at the origin follows. 
	
	\begin{proof}[Proof of Theorem~\ref{secondorderresult}]
		Since $a_{0}<4\pi$ and $a_{0}\phi(x)=\frac{1}{2\pi}a_{0}\left(\ln5-\ln|x|\right)$, by Lemma ~\ref{lemma2.2}, there exists a constant $b:=-\frac{1}{2\pi}a_{0} >-2$, such that $u$ satisfies \begin{equation}\label{K(x) second-order res}
			u(x)=v(x)+h(x)+b\ln|x| \quad\text{for }x\in B_{1} \setminus\{0\},
		\end{equation} 
		where  $v$ is given by \eqref{v(x)} and $h \in C^{\infty}(B_{1})$ is a harmonic function in $B_{1}$. Here $h$ has been redefined by absorbing the constant $-b\ln 5$. We will prove that if $\alpha \in \left(0,2\right)$ and $b > \frac{\alpha-4}{2}$, then $v \in C_{loc}^{\gamma} (B_{1})$ for some $0< \gamma <1$. By \eqref{K(x) second-order res}, we know that $v$ is a non-negative solution of 
		\begin{align*}
			-\Delta v(x) &= K(x)\left(\frac{1}{|x|^\alpha}*e^u\right)e^{u(x)} \\
			&=K(x)\left(\int_{B_{1}\setminus\{0\}}\frac{e^{v(y)}e^{h(y)}|y|^{b}}{|x-y|^\alpha} \ud y\right)e^{v(x)}e^{h(x)}|x|^{b}   \quad \text{in } B_{1}
		\end{align*}
		in the sense of distributions. According to $K\in L^{\infty}(B_{1})$ and \eqref{second-order assumption}, we have
		\begin{equation*}
			\int_{B_{1} \setminus \{0\}}K(x)\left(\int_{B_{1}\setminus\{0\}}\frac{e^{v(y)}e^{h(y)}|y|^{b}}{|x-y|^\alpha} \ud y\right)e^{v(x)}e^{h(x)}|x|^{b} \ud x <\infty.
		\end{equation*} 
		
		We claim that  $K(x)\left(\frac{1}{|x|^\alpha}*e^u\right)e^{u(x)} \in L^{p_{0}}_{loc}(B_{1})$ for some $p_{0}>1$. In fact, we fix an arbitrary $r \in (0,1)$. For convenience, we define $f(x):=e^{u(x)}=|x|^{b}e^{h(x)}e^{v(x)}$ and $g_{r}(x) := \int_{B_{r} \setminus \{0\}}\frac{f(y)}{|x-y|^\alpha} \ud y$. We divide $K(x)\left(\frac{1}{|x|^\alpha}*e^u\right)e^{u(x)}$ into two parts:
		\begin{equation*}
		\begin{aligned}
			K(x)\left(\frac{1}{|x|^\alpha}*e^u\right)e^{u(x)} =&K(x)g_{r}(x)f(x)+K(x) \left(\int_{B_{r_{1}}\setminus B_{r}} \frac{e^{u(y)} }{|x-y|^{\alpha}}\ud y \right) f(x) \\
			&+K(x) \left(\int_{B_{1}\setminus B_{r_{1}}} \frac{e^{u(y)} }{|x-y|^{\alpha}}\ud y \right) f(x) \\
			:= &\mathcal{I}_{1}(x) +\mathcal{I}_{2}(x)+\mathcal{I}_{3}(x) \quad \text{in }B_{r},
		\end{aligned}
			\end{equation*}
		where we choose $r_{1}$ satisfying $r<r_{1}<1$.
		
		For $\mathcal{I}_{1}$, by the results in \cite[Corollary 1 and Remark 2]{Brezis-Merle}, we konw that $e^{v} \in L^{p}(B_{r})$ for each $p>0$. Since $b>-2$, there exists $p_{f}>1$, such that $bp_{f}>-2$, whence $|x|^{b} \in L^{p_{f}}(B_{r})$. We consider two cases:\\
		(i) $b \geq 0$;\\
		(ii)$-2<b<0$.
		
			For case (i), we can choose any real number $p_{f}>1$ such that $|x|^{b} \in L^{p_{f}}(B_{r})$. Therefore, $f \in L^{q}(B_{r})$ for any $q>1$. Since  $0<\alpha<2$, by the Hardy–Littlewood–Sobolev inequality, we easily get that there exists some $q_{1}>1$ such that $g_{r}\in L^{q_{1}} (B_{r})$, and hence $\mathcal{I}_{1}(x) \in L^{q_{1}} (B_{r})$. 
		
		For case (ii), by the H\"{o}lder inequality, we can choose some $p_{1}$ satisfying $1<p_{1}<\min\{\frac{2}{2-\alpha}, -\frac{2}{b}\}$, such that  $|x|^{b} \in L^{p_{1}}(B_{r})$ and $f \in L^{p_{1}}(B_{r})$. Applying the Hardy–Littlewood–Sobolev inequality together with the chosen value of $p_{1}$, we obtain the existence of $t>1$ satisfying
		\begin{equation}\label{1/t}
			\frac{1}{t} =\frac{1}{p_{1}}+\frac{\alpha}{2}-1,
		\end{equation}
		such that
		\begin{equation}\label{hls}
			\|g_{r}\|_{L^{t}(B_{r})} \leq C\|f\|_{L^{p_{1}}(B_{r})}
		\end{equation}
		for some constant $C>0$. Next, we consider two cases here:\\
		(a) $-\frac{2}{b}\geq\frac{2}{2-\alpha}$, i.e. $\alpha-2\leq b<0$,\\
		(b) $-\frac{2}{b}<\frac{2}{2-\alpha}$, i.e. $-2<b<\alpha-2$.
		
		For case (a), we have $1< p_{1}<\frac{2}{2-\alpha}$. Thus, combined with \eqref{1/t}, we know that $0<\frac{1}{t}<\frac{\alpha}{2}$. By direct computation, for each fixed $b$, there exists a real number $s$ such that $\frac{t}{t-1}<s<-\frac{2}{b}$ when $t$ is sufficiently large, and $|x|^{b} \in L^{s}(B_{r}) $ as well as $f \in L^{s} (B_{r})$. Combining this with \eqref{hls}, one can find some $q_{2}>1$ satisfying $\frac{1}{s} + \frac{1}{t} = \frac{1}{q_{2}}$, so that 
		\begin{equation*}
			\int_{B_{r} \setminus \{0\}}\left|f(x)g_{r}(x)\right|^{q_{2}} \ud x \leq \int_{B_{r} \setminus \{0\}} \left|f(x)\right|^{q_{2}} \left|g_{r}(x)\right|^{q_{2}} \ud x \leq \|f^{q_{2}}\|_{\frac{s}{q_{2}}} \|(g_{r})^{q_{2}}\|_{\frac{t}{q_{2}}}<\infty,
		\end{equation*}
		where we have used the H\"{o}lder inequality. Consequently, it follows that
		\begin{equation*}
			\mathcal{I}_{1}(x) \in L^{q_{2}}(B_{r}).
		\end{equation*}
		
		For case (b), we have $1<p_{1}<-\frac{2}{b}$. Under the assumption that $0<\alpha<2$ and $b>\frac{\alpha-4}{2}$, we obtain that $\frac{\alpha-4}{2}<b<\alpha-2$ and $\frac{-b+\alpha-2}{2}<\frac{1}{t}<\frac{\alpha}{2}$. Hence, for each fixed $b$, we can also find a real number $s$ satisfying $\frac{t}{t-1}<s<-\frac{2}{b}$ such that  $|x|^{b} \in L^{s}(B_{r}) $ and $f \in L^{s} (B_{r})$. Using a similar method as employed in case (a), we conclude that there exists some $q_{3}>1$ satisfying $\frac{1}{s} + \frac{1}{t} = \frac{1}{q_{3}}$ such that 
		\begin{equation*}
			\mathcal{I}_{1}(x) \in L^{q_{3}}(B_{r}).
		\end{equation*}
		
		For $\mathcal{I}_{2}$, by the regularity of $u$ in $\bar{B}_{r_{1}} \setminus \left\{0\right\}$, we know that there exists a constant $M>0$ such that $u<M$ in $\bar{B}_{r_{1}} \setminus B_{r}$. Consequently, we have the estimate 
		\begin{align*}
			\mathcal{I}_{2}(x) &\leq C \left(\int_{B_{r_{1}}\setminus B_{r}} \frac{e^{M}}{|x-y|^{\alpha}} \ud y\right)f(x) \\
			& \leq C \left(\int_{B_{2}} \frac{1}{|y|^{\alpha}} \ud y\right)f(x) \\
			&\leq Cf(x)  \quad \text{for }x \in B_{r}.
		\end{align*}
		By the similiar argument used in the analysis of $\mathcal{I}_{1}$, we know that there exists some $q_{4}=p_{1}>1$ such that $\mathcal{I}_{2} \in L^{q_{4}}(B_{r})$.
		
		For $\mathcal{I}_{3}$, by the assumption \eqref{second-order assumption} and $\alpha \in(0,n)$, we have
		\begin{equation*}
			\left(\int_{B_{1} \setminus \{0\}} e^{u(x)} \ud x\right)^{2} \leq 2^{\alpha}\int_{B_{1}\setminus\{0\}}\left(\frac{1}{|x|^\alpha}*e^u\right)e^{u(x)} \ud x <\infty,
		\end{equation*}
		which implies that $e^{u} \in L^{1}(B_{1} \setminus \left\{0\right\})$. For $x \in B_{r} $ and $y \in B_{1} \setminus B_{r_{1}}$, we observe that 
		\begin{equation*}
			|x-y|\geq |y|-|x| \geq r_{1}-r>0.
		\end{equation*}
		Therefore, we get
		\begin{align*}
			\mathcal{I}_{3}(x) &\leq C\left(\int_{B_{1}\setminus B_{r_{1}}} \frac{e^{u(y)} }{|x-y|^{\alpha}}\ud y \right) f(x) \\
			&\leq C \left(r_{1} -r\right)^{-\alpha} \left(\int_{B_{1}\setminus B_{r_{1}}} e^{u(y)} \ud y\right)f(x)  \\
			& \leq Cf(x) \quad \text{for }x \in B_{r}.
		\end{align*}
	By an argument analogous to that for  $\mathcal{I}_{1}$,
		we have $\mathcal{I}_{3} \in L^{q_{4}}(B_{r})$.
		
		To summarize, for $\alpha \in \left(0,2\right)$ and $b > \frac{\alpha-4}{2}$, we can select $p_{0} =\min\left\{q_{1},q_{2},q_{3},q_{4}\right\}$, then $K(x)\left(\frac{1}{|x|^\alpha}*e^u\right)e^{u(x)} \in L^{p_{0}}(B_{r})$. Since $r \in(0,1)$ is arbitrary, we deduce that
		\begin{equation*}
			K(x)\left(\frac{1}{|x|^\alpha}*e^u\right)e^{u(x)} \in L^{p_{0}}_{loc}(B_{1}) \quad\text{for some } p_{0}>1.
		\end{equation*} 
		An application of the standard $W^{2,p}$ elliptic regularity theory (see \cite{Gilbarg-Trudinger}) directly yields the improved regularity $v \in W^{2,p_{0}}_{loc}(B_{1})$. It then follows from the Sobolev embedding theorem that $v$ is locally H\"older continuous, i.e., $v \in C^\gamma_{loc}(B_{1})$ for some $\gamma \in (0,1)$.
	\end{proof}
	\subsection{ Even-dimensional case ($n \geq 4$)}
In this subsection,	 following a similar detailed analysis in two-dimensional case, we study the Hartree-type equation \eqref{k(x)equation} in even dimensions $n \geq 4$. The corresponding polyharmonic equations involve higher-order singularities, which introduce substantial new challenges. Thus, we need an  additional decay condition \eqref{o(r^{n-2})} near the origin. Under the finite total curvature assumption \eqref{second-order assumption} and condition \eqref{o(r^{n-2})}, we also derive a representation formula for solutions to  \eqref{k(x)equation} in the even-dimensional case, and further investigate the asymptotic behavior of these solutions near the origin.
	
	If $n\geq4$ is even, a fundamental solution of $\left(-\Delta\right)^{\frac{n}{2}}$ is given by $\phi(x) = c_{n} \ln\frac{5}{|x|}$, with $c_{n}>0$ being a constant. Then we have
	\begin{lemma}\label{even case expression}
		Let $n\geq4$ be an even integer and $\alpha>0$. Suppose that $u \in C^{n}\left(B_{1}\setminus\{0\}\right)$ is a solution of \eqref{k(x)equation} satisfying condition \eqref{second-order assumption} and \eqref{o(r^{n-2})}, and $K \in L^{\infty}(B_{1})$ is non-negative. Then u decomposes as
		\begin{equation}
			u(x)=v(x)+h(x)+ a_{0}\phi(x)  \quad\text{for }x\in B_{1} \setminus\{0\},
		\end{equation}
		where $v$ is defined by \eqref{v(x)}. The function $h \in C^{\infty}(B_{1})$ is a polyharmonic function in $B_{1}$ satisfying $(-\Delta)^{\frac{n}{2}}h =0$, and the constant $a_{0}$ satisfies the constraint  $a_{0}<\frac{n}{c_{n}}$.
	\end{lemma}
	\begin{proof}
		If $u$ solves \eqref{k(x)equation} and $K \in L^{\infty}(B_{1})$, Proposition~\ref{representation formula} with $s=0$ gives the decomposition:
		\begin{equation}\label{proof u(x) decomposition}
			u(x)=v(x)+h(x)+\sum_{|\beta|\leq n-1} a_{\beta}D^{\beta}\phi(x)  \quad\text{for }x\in B_{1} \setminus\{0\},
		\end{equation}
		where $v$ is given by \eqref{v(x)} and $h \in C^{\infty}(B_{1})$ is a solution of $(-\Delta)^{\frac{n}{2}}h =0$ in $B_{1}$.
		Next we will prove that $\sum_{1\leq|\beta|\leq n-1} a_{\beta}D^{\beta}\phi(x) \equiv 0$. We can split the proof in two steps.

		\textbf{Step 1.}
			We show that all coefficients corresponding to derivatives of order $2 \leq |\beta| \leq n-1$ must be zero.
		
		By the Fubini Theorem and the condition \eqref{second-order assumption}, we obtain
		\begin{align*}
			\int_{B_{r}}v(x) \ud x &\leq C \int_{B_{r}} \left[\int_{B_{1}} \ln \frac{5}{|x-y|}K(y) \left(\frac{1}{|y|^{\alpha}}\ast e^{u}\right)e^{u(y)} \ud y\right]\ud x\\
			&\leq C \int_{B_{1}} \left(\int_{B_{r}} \ln \frac{5}{|x-y|} \ud x\right) K(y) \left(\frac{1}{|y|^{\alpha}}\ast e^{u}\right)e^{u(y)} \ud y\\
			& \leq C r^{n}\ln \frac{1}{r} \int_{B_{1}} K(y) \left(\frac{1}{|y|^{\alpha}}\ast e^{u}\right)e^{u(y)} \ud y \\
			& \leq Cr^{n}\ln \frac{1}{r}, \qquad \text{for } 0<r<\frac{1}{2}.
		\end{align*}
		Combining this with \eqref{o(r^{n-2})} and \eqref{proof u(x) decomposition}, we conclude that
		\begin{equation}\label{Pn-1/|x|}
			\int_{B_{r}} \left|\sum_{|\beta|\leq n-1} a_{\beta}D^{\beta}\phi(x)\right| \ud x =o(1) r^{n-2} \quad \text{as }r \to 0.
		\end{equation}
	
		For the top-order derivatives $|\beta|=n-1$, we group them as 
		\begin{equation}
			\sum_{|\beta|= n-1} a_{\beta}D^{\beta}\phi(x) =\frac{P_{n-1}(x)}{|x|^{2n-2}},
		\end{equation}
		where $P_{n-1}(x)$ is a homogeneous polynomial of degree $n-1$. Restrict $P_{n-1}$ to the unit sphere $\mathbb{S}^{n-1}$ by the map $\theta \mapsto P_{n-1}(\theta) $. This restriction is a continuous function on the compact set $\mathbb{S}^{n-1}$.  We claim that $P_{n-1} \equiv0$. In fact, 
		if $P_{n-1} \not\equiv 0$, then there exists at least one point $\theta_{0} \in \mathbb{S}^{n-1}$ such that $P_{n-1}(\theta_{0})>0$. Let $d_{0} = \frac{1}{2}P_{n-1}(\theta_{0})$. Therefore, by continuity, there exists a geodesic ball neighborhood $U_{0} \subset \mathbb{S}^{n-1}$ of $\theta_{0}$ such that $\left|P_{n-1}\right| \geq d_{0}>0$ for all $\theta \in U_{0}$. Using polar coordinates $x=\rho\theta$ with $\theta \in \mathbb{S}^{n-1}$, the volume element becomes 
		\begin{equation*}
			\ud x = \rho^{n-1} \ud \rho \ud \theta.
		\end{equation*}
		By homogeneity, for $\theta \in \mathbb{S}^{n-1} $ and $\rho>0$, we have 
		\begin{equation*}
			P_{n-1}(x)=\rho^{n-1}P_{n-1}(\theta).
		\end{equation*}
		Now, define the region $V_{r}=[0,r] \times U_{0}$. Then,
		\begin{equation*}
			\int_{V_{r}} \left|	\sum_{|\beta|= n-1} a_{\beta}D^{\beta}\phi(x)\right| \ud x =\int_{U_{0}} \left|P_{n-1}(\theta)\right| \ud \theta \int_{0}^{r} \ud \rho \geq d_{1} r
		\end{equation*}
		for some constant $d_{1}>0$. On the other hand, a direct computation gives
		\begin{equation*}
			\int_{B_{r}} \left|\sum_{|\beta|\leq n-2} a_{\beta}D^{\beta}\phi(x)\right| \ud x \leq C \int_{B_{r}} |x|^{-n+2} \ud x \leq Cr^{2}.
		\end{equation*}
		Thus, combining these two estimates, we have 
		\begin{align*}
			\int_{V_{r}} \left|\sum_{|\beta|\leq n-1} a_{\beta}D^{\beta}\phi(x)\right| \ud x &\geq \int_{V_{r}} \left|\sum_{|\beta|= n-1} a_{\beta}D^{\beta}\phi(x)\right| \ud x -\int_{V_{r}} \left|\sum_{|\beta|\leq n-2} a_{\beta}D^{\beta}\phi(x)\right| \ud x\\
			& \geq \int_{V_{r}} \left|\sum_{|\beta|= n-1} a_{\beta}D^{\beta}\phi(x)\right| \ud x -\int_{B_{r}} \left|\sum_{|\beta|\leq n-2} a_{\beta}D^{\beta}\phi(x)\right| \ud x\\
			& \geq d_{1}r-Cr^{2}\\
			& \geq \frac{d_{1}}{2}r \quad \text{for small } r>0.
		\end{align*}
		This contradicts \eqref{Pn-1/|x|}. Hence, we conclude that $P_{n-1} \equiv0$, which implies
		\begin{equation*}
			\sum_{|\beta|= n-1} a_{\beta}D^{\beta}\phi(x) =0.
		\end{equation*}
		By iterating this procedure, we conclude that
		\begin{equation*}
			\sum_{|\beta|= j} a_{\beta}D^{\beta}\phi(x) =0,
		\end{equation*}
		for every index $2 \leq j \leq n-2$.
		
		\textbf{Step 2.} We show that $\sum_{|\beta|= 1} a_{\beta}D^{\beta}\phi(x) =0$ and $a_{0} <\frac{n}{c_{n}}$.
		
		By \eqref{second-order assumption}, \eqref{proof u(x) decomposition} and the results in step 1, we have
		\begin{equation}\label{higher-order assumption1}
			\int_{B_{\frac{1}{2}} \setminus \{0\}}\left(\int_{B_{\frac{1}{2}} \setminus \{0\}}\frac{e^{v(y)}e^{h(y)}e^{\sum_{|\beta|\leq1}a_{\beta}D^{\beta}\phi(y)}}{|x-y|^{\alpha}}\ud y\right)e^{v(x)}e^{h(x)}e^{\sum_{|\beta|\leq1}a_{\beta}D^{\beta}\phi(x)} \ud x<+\infty.
		\end{equation}
		To exploit this, we first note that $h \in C^{\infty}(B_{1}(0))$ implies that $C_{1}< e^h <C_{2}$ in $B_{1/2}$ for two constants $C_{1}, C_{2}>0$. By the nonnegativity of $K$, $e^{v} \geq 1 $ in $B_{1/2} \setminus \{0\}$. For $x,y \in B_{1/2}$, the kernel satisfies $\frac{1}{|x-y|^\alpha} \geq 1$.  Combining these bounds with \eqref{higher-order assumption1}, we obtain the finiteness of the integral involving the remaining exponential term:
		\begin{equation*}
			\int_{B_{\frac{1}{2}} \setminus \{0\}}e^{\sum_{|\beta|\leq 1} a_{\beta} D^{\beta}\phi(x) }\ud x<+\infty.
		\end{equation*}
		Namely, 
		\begin{equation}\label{higherorder assumption 2}
			\int_{B_{\frac{1}{2}} \setminus \{0\}} \left(\frac{5}{|x|}\right)^{c_{n}a_{0}}e^{-c_{n}\sum_{i=1}^{n}a_{i}\frac{x_{i}}{|x|^{2}}} \ud x <+\infty,
		\end{equation}
		where $a_{i} = a(0, \cdots, 1,\cdots,0)$. An argument parallel to the one used in proving Theorem~\ref{secondorderresult}, we have  $a_{i} =0$ for every $i=1,2,\cdots,n$, and hence $\sum_{|\beta|= 1} a_{\beta}D^{\beta}\phi(x) =0$. From \eqref{higherorder assumption 2}, it follows that
		\begin{equation*}
			\int_{B_{\frac{1}{2}} \setminus \{0\}}\left(\frac{5}{|x|}\right)^{c_{n}a_{0}} \ud x <+\infty,
		\end{equation*}
		which implies  $c_{n} a_{0}<n$. We  complete the proof.
	\end{proof}
	Now, we discuss the regularity of $v$. Before that, we need the following lemma to estimate $v$. 
	\begin{lemma}[Exponential integrability for logarithmic potentials]\label{Exponential integrability}
		Let $F \ge 0$ and $F \in L^{1}(B_{1})$. Define
		\begin{equation*}
			v(x)=\int_{B_1} K(x,y)F(y) \ud y, \quad \text{in }B_1 \subset \mathbb{R}^{n},
		\end{equation*}
		where the kernel $K$ satisfies
		\begin{equation}\label{kernel-bound}
			|K(x,y)| \le C_{1}\bigl(1 + |\ln|x-y||\bigr), \quad x,y \in B_1,
		\end{equation}
		for some constant $C_{1}>0$. Then for any $p>1$, we have
		\begin{equation*}
			e^{pv} \in L^{1}(B_{1}).
		\end{equation*}
	\end{lemma}
	\begin{remark}\label{Exponential integrability remark}
		Lemma \ref{Exponential integrability} does not require $F$ to be finite everywhere. Any non-negative $F\in L^{1}(B_{1})$ is admissible, even if $F$ has isolated singularities. Since $F$ is integrable, the potential
		\begin{equation*}
			v(x)=\int_{B_1} K(x,y)F(y) \ud y
		\end{equation*}
		is well defined almost everywhere in $B_{1}$, which is sufficient to get  $e^{pv}\in L^{1}(B_{1})$.
	\end{remark}
	\begin{proof}
		Since $F \ge 0$ and $F \in L^{1}(B_{1})$, for any $\varepsilon>0$, we may choose 
		$\lambda>0$ such that the truncation
		\begin{equation*}
			F_{1} =F\chi_{\left\{F<\lambda\right\}} \qquad \text{and}  \qquad 	F_{2} =F\chi_{\left\{F\geq\lambda\right\}}
		\end{equation*}
		satisfying $F_{1} \in L^{\infty}(B_{1})$ and $\|F_{2}\|_{L^{1}(B_{1})} <\varepsilon$, where cut-off function $\chi_{\Omega} $ is defined by
		\begin{equation*}
			\chi_{\Omega}=
			\begin{cases}
				\displaystyle 1, &  \quad\text{in } \Omega,\\[6pt]
				0, & \quad\text{in }\Omega^{c}.
			\end{cases}
		\end{equation*}
		Therefore, for any arbitrarily fixed $p>1$, there exists $\lambda_{p}>0$ such that $F_{1} \in L^{\infty}(B_{1})$ and $\|F_{2}\|_{L^{1}(B_{1})} < \frac{n}{2pC_{1}}$. Then the function $v$ admits a decomposition $v=v_{1}+v_{2}$, in which
		\begin{equation*}
			v_{i}(x)=\int_{B_{1}} K(x,y) F_{i}(y) \ud y \qquad i=1,2.
		\end{equation*}
		By the H\"older inequality, we have
		\begin{align*}
			\|v_{1}\|_{L^{\infty}(B_{1})} &\leq \|K(x,y)\|_{L^{1}(B_{1})} \|F_{1}\|_{L^{\infty}(B_{1})}\\
			& \leq C_{1}\|F_{1}\|_{L^{\infty}(B_{1})}\int_{B_{2}}\left(1+\left|\ln|y|\right|\right) \ud y \\
			& <\infty,
		\end{align*}
		which implies that $e^{pv_{1}} \in L^{1}(B_{1})$. On the other hand, using \eqref{kernel-bound} and Jensen's inequality, we have
		\begin{align*}
			\int_{B_{1}} e^{p|v_{2}(x)|} \ud x &\leq \int_{B_{1}} e^{pC_{1}\int_{B_{1}} \left(1+\left|\ln|x-y|\right|\right)F_{2}(y) \ud y} \ud x\\
			&\leq C \int_{B_{1}}e^{pC_{1}\|F_{2}\|_{1}\int_{B_{1}}\left(\left|\ln|x-y|\right|\frac{F_{2}(y)}{\|F_{2}\|_{1}}\right)\ud y} \ud x\\
			& \leq C \int_{B_{1}} \int_{B_{1}}\left(e^{\left|\ln|x-y|\right|}\right)^{\frac{n}{2}}\frac{F_{2}(y)}{\|F_{2}\|_{1}} \ud y \ud x\\
			& \leq C \int_{B_{1}} \left(\frac{1}{|y|}\right)^{\frac{n}{2}} \ud y + C \int_{B_{2} \setminus B_{1}} |y|^{\frac{n}{2}} \ud y\\
			& < \infty.
		\end{align*}
		Thus, we conclude that for any $p>1$, $e^{pv} \in L^{1}(B_{1})$. This completes the proof.
	\end{proof}
	Now we begin to prove Theorem~\ref{even result case}.
	\begin{proof}[Proof of Theorem~\ref{even result case}]
		We define
		\begin{equation*}
			b:=-c_{n}a_{0}>-n.
		\end{equation*}
		Since $a_{0}\phi(x)=-b\ln5+b\ln|x|$, by Lemma~\ref{even case expression}, we have
		\begin{equation}\label{even result case inequality1}
			u(x)=v(x)+h(x)+b\ln|x|.
		\end{equation}
		Here $v$ is given by \eqref{v(x)} and $h \in C^{\infty}(B_{1})$ is a solution of
		$(-\Delta)^{\frac{n}{2}}h =0$ in $B_{1}$, where we have absorbed the constant
		$-b\ln 5$ into the harmonic function $h$. Therefore, we know that $v$ is a nonnegative solution of
		\begin{equation*}
			\left(-\Delta\right)^{\frac{n}{2}}v=K(x)\left(\frac{1}{|x|^\alpha}*e^u\right)e^{u(x)} \quad \text{in }B_{1}
		\end{equation*}
		in the sense of distributions. Using $K \in L^{\infty}(B_{1})$ and assumption \eqref{second-order assumption}, we obtain that
		\begin{equation*}
			\int_{B_{1} \setminus \{0\}} K(x)\left(\frac{1}{|x|^\alpha}*e^u\right)e^{u(x)} \ud x <+\infty.
		\end{equation*}
		Since $c_{n} \ln \frac{5}{|x-y|}$ satisfies \eqref{kernel-bound}, then by Lemma~\ref{Exponential integrability} and Remark~\ref{Exponential integrability remark}, we have $e^{v} \in L^{p}(B_{1})$ for any $p>1$. 
		
		Next, we will prove that $v \in C^{\gamma}_{loc} (B_{1})$ for some $0<\gamma<1$.
			By \eqref{k(x)equation} and \eqref{even result case inequality1}, we know that $v$ is a non-negative solution of 
		\begin{align*}
			(-\Delta)^{\frac{n}{2}} v(x) &= K(x)\left(\frac{1}{|x|^\alpha}*e^u\right)e^{u(x)} \\
			&=K(x)\left(\int_{B_{1}\setminus\{0\}}\frac{e^{v(y)}e^{h(y)}|y|^{b}}{|x-y|^\alpha} \ud y\right)e^{v(x)}e^{h(x)}|x|^{b}   \quad \text{in } B_{1}
		\end{align*}
		in the sense of distributions. Proceeding in a manner similar to the proof of Theorem~\ref{secondorderresult}, one finds that, provides $\alpha \in (0,n)$ and $b>\frac{\alpha-2n}{2}$, it follows that
		\begin{equation*}
			K(x)\left(\int_{B_{1}\setminus\{0\}}\frac{e^{v(y)}e^{h(y)}|y|^{b}}{|x-y|^\alpha} \ud y\right)e^{v(x)}e^{h(x)}|x|^{b} \in L^{p_{0}}(B_{1})
		\end{equation*}
		for some $p_{0}>1$. A subsequent application of the standard $W^{n,p}$ elliptic regularity theory yields $v\in W^{n,p_{0}}_{loc}(B_{1})$. It then follows from the Sobolev embedding theorem that $v \in C^{\gamma}_{loc}(B_{1})$ for some $0<\gamma<1$. 
	Thus, by  \eqref{even result case inequality1}, we have
		\begin{equation*}
			u(x)=\varphi(x)+b\ln|x|  \quad \text{in }B_{1} \setminus\left\{0\right\},
		\end{equation*}
		in which $\varphi:=v+h \in C^{\gamma}_{loc}(B_{1})$ for some $0<\gamma<1$. We finish the proof.
	\end{proof}
	\section{Asymptotic behavior at isolated singularities in odd dimension case} 
	In this section, we study the isolated singularity problem for the Hartree-type equation \eqref{k(x)equation} with odd dimensions $n \ge 3$. In this case, the operator $(-\Delta)^{\frac{n}{2}}$ is a nonlocal operator, and we mainly consider the distributional solutions defined in the space $u\in\mathcal{L}_{\frac{n}{2}}(\mathbb{R}^{n})$. Under the finite total curvature condition \eqref{second-order assumption} and the decay condition \eqref{o(r^{n-2})}, we also show that the asymptotic behavior near the origin has the same form as in the even-dimensional case. The following lemma provide the  corresponding representation formula for the solutions to \eqref{k(x)equation} in odd dimensions.
	
	\begin{lemma}
		Let $n\geq 3$ be an odd integer and $\alpha>0$. Suppose that $u \in \mathcal{L}_{\frac{n}{2}} (\mathbb{R}^{n}) \cap C^{n} \left(B_{1} \setminus \left\{0\right\}\right)$ is a solution of \eqref{k(x)equation} satisfying condition \eqref{second-order assumption} and \eqref{o(r^{n-2})}, and $K \in L^{\infty}(B_{1})$ is non-negative. Then $u$ decomposes as
		\begin{equation*}
			u(x)=v(x)+h(x)+b\ln|x| \quad \text{for }x \in B_{1} \setminus\left\{0\right\},
		\end{equation*}
		where $v$ is defined by \eqref{v(x)}, $h \in C^{\infty}(B_{1})$ is a polyharmonic function in $B_{1}$ satisfying $(-\Delta)^{\frac{n}{2}}h =0$, and the constant $b$ satisfies $b>-n$.
	\end{lemma}
	\begin{proof}
		Let $v$ be defined as in \eqref{v(x)}. Standard potential estimates imply 
		$v\in \mathcal{L}_{\frac{n}{2}}(\mathbb{R}^n)$, and 
		\[
		(-\Delta)^{\frac{n}{2}}v(x)=K(x)\left(\frac{1}{|x|^\alpha}*e^u\right)e^{u(x)}\qquad\text{in }\mathcal D'(B_1).
		\]
		Set \(w=u-v\). Then
		\[
		(-\Delta)^{\frac n2}w=0 \qquad\text{in } B_1\setminus\{0\}.
		\]
		Hence, for any \(\psi\in C_c^\infty(B_1\setminus\{0\})\), we have
		\[
		\langle(-\Delta)^{\frac n2}w,\psi\rangle=0,
		\]
		which implies that
		\[
		\mathrm{supp}\big((-\Delta)^{\frac n2}w\big)\subset\{0\}.
		\]
		By the classical structure theorem of Schwartz, we know that
		\[
		(-\Delta)^{\frac n2}w=\sum_{|\beta|\leq M}a_\beta D^\beta\delta_0,
		\qquad a_\beta\in\mathbb{R}.
		\]
		Using \( (-\Delta)^{\frac n2}\phi=\delta_0 \), we have
		\[
		w(x)
		=h(x)+\sum_{|\beta|\leq M}a_\beta D^\beta\phi(x),
		\]
		where \(h\) satisfies \( (-\Delta)^{\frac n2}h=0 \) in \(B_1\) and $\phi(x) =c_{n}\ln \frac{5}{|x|}$ is a fundamental solution of $(-\Delta)^{\frac{n}{2}}$. By the elliptic regularity, we have \( h\in C^\infty(B_1) \). Since \(u,v,h\in L^1_{\mathrm{loc}}(B_1)\) and 
		\(D^\beta\phi(x)\sim |x|^{-|\beta|}\),
		the local integrability near \(0\)  requires \( |\beta|<n \).
		Thus \(a_\beta=0\) whenever $ |\beta|\geq n $, and $M\leq n-1$. Therefore, we conclude that
		\begin{equation*}
			u(x)=v(x)+h(x)+\sum_{|\beta|\leq n-1} a_{\beta}D^{\beta}\phi(x)  \quad\text{for }x\in B_{1} \setminus\{0\}.
		\end{equation*}
		
		Arguing as in the proof of Lemma~\ref{even case expression}, we obtain that 
		\begin{equation*}
			u(x)=v(x)+h(x)+ b\ln|x|  \quad\text{for }x\in B_{1} \setminus\{0\},
		\end{equation*}
		where $b>-n$ and $h$ has been redefined by absorbing a constant into the harmonic part. 
	\end{proof}
	Now we discuss the regularity of $v$ and prove Theorem~\ref{odd case result}.
	\begin{proof}[Proof of Theorem~\ref{odd case result}]
		Since $c_{n} \ln\frac{5}{|x-y|}$ satisfies \eqref{kernel-bound}, then by Lemma~\ref{Exponential integrability} and Remark~\ref{Exponential integrability remark}, we have $e^{pv} \in L^{1}(B_{1})$ for any $p>1$. By \eqref{k(x)equation} and the expression of $u$, we know that $v$ is a non-negative solution of 
		\begin{align*}
			(-\Delta)^{\frac{n}{2}} v(x) &= K(x)\left(\frac{1}{|x|^\alpha}*e^u\right)e^{u(x)} \\
			&=K(x)\left(\int_{B_{1}\setminus\{0\}}\frac{e^{v(y)}e^{h(y)}|y|^{b}}{|x-y|^\alpha} \ud y\right)e^{v(x)}e^{h(x)}|x|^{b}   \quad \text{in } B_{1}
		\end{align*}
		in the sense of distributions. An argument analogous to the proof of Theorem~\ref{secondorderresult} shows that, provided $\alpha \in (0,n)$ and $b>\frac{\alpha-2n}{2}$, we have
		\begin{equation*}
			K(x)\left(\int_{B_{1}\setminus\{0\}}\frac{e^{v(y)}e^{h(y)}|y|^{b}}{|x-y|^\alpha} \ud y\right)e^{v(x)}e^{h(x)}|x|^{b} \in L^{p_{0}}(B_{1})
		\end{equation*}
		for some $p_{0}>1$. By the expression of $v$ and the classical potential estimate (see, e.g., Gilbarg-Trudinger \cite{Gilbarg-Trudinger}), we obtain that $v \in C^{\gamma}_{loc}(B_{1})$ for some $0<\gamma<1$. Consequently, we get
		\begin{equation*}
			u(x)=\varphi(x)+b\ln|x|  \quad \text{in }B_{1} \setminus\left\{0\right\},
		\end{equation*}
		in which $\varphi:=v+h \in C^{\gamma}_{loc}(B_{1})$ for some $0<\gamma<1$.
	\end{proof}

\medskip
\medskip
\medskip
\textbf{Declaration}: {The authors declared that they have no conflict of interest.}

\end{document}